\documentclass[dvipsnames, a4paper, 12pt]{article}

\usepackage[utf8]{inputenc}
\usepackage[T1]{fontenc}
\usepackage[a4paper, margin=2.5cm]{geometry}
\usepackage{needspace}

\usepackage{libertine} 
\usepackage{inconsolata} 
\usepackage{textpos}

\usepackage{amsmath, amsthm, amssymb}
\usepackage{graphicx}
\usepackage{enumerate}
\usepackage{enumitem}
\usepackage{authblk}
\usepackage[textsize=footnotesize, textwidth=2cm, color=yellow]{todonotes}
\usepackage{thmtools}
\usepackage{thm-restate}
\usepackage[colorlinks=true, citecolor=red]{hyperref}
\usepackage{float}
\usepackage[nameinlink]{cleveref}
\usepackage{xspace}
\usepackage{array}
\usepackage{mathtools}
\usepackage{xfrac}

\renewenvironment{abstract}
{\small\vspace{-1em}
\begin{center}
\bfseries\abstractname\vspace{-.5em}\vspace{0pt}
\end{center}
\list{}{
\setlength{\leftmargin}{0.6in}%
\setlength{\rightmargin}{\leftmargin}}%
\item\relax}
{\endlist}

\declaretheorem[name=Theorem, numberwithin=section]{theorem}
\crefformat{theorem}{#2Theorem~#1#3}
\Crefformat{theorem}{#2Theorem~#1#3}
\declaretheorem[name=Lemma, sibling=theorem]{lemma}
\crefformat{lemma}{#2Lemma~#1#3}
\Crefformat{lemma}{#2Lemma~#1#3}

\declaretheorem[name=Corollary, sibling=theorem]{corollary}
\crefformat{corollary}{#2Corollary~#1#3}
\Crefformat{corollary}{#2Corollary~#1#3}

\declaretheorem[name=Claim, sibling=theorem]{claim}
\crefformat{claim}{#2Claim~#1#3}
\Crefformat{claim}{#2Claim~#1#3}
\crefformat{equation}{(#2#1#3)}
\Crefformat{equation}{(#2#1#3)}

\crefformat{figure}{#2Figure~#1#3}
\Crefformat{figure}{#2Figure~#1#3}
\declaretheorem[name=Observation, sibling=theorem]{observation}
\crefformat{observation}{#2Observation~#1#3}
\Crefformat{observation}{#2Observation~#1#3}

\def\cqedsymbol{\ifmmode$\lrcorner$\else{\unskip\nobreak\hfil
\penalty50\hskip1em\null\nobreak\hfil$\lrcorner$
\parfillskip=0pt\finalhyphendemerits=0\endgraf}\fi}
\newcommand{\cqed}{\renewcommand{\qed}{\cqedsymbol}}

\def\cH{\mathcal{H}}
\def\cP{\mathcal{P}}
\def\cS{\mathcal{S}}
\def\Pp{\cP}
\def\cQ{\mathcal{Q}}

\def\cS{\mathcal{S}}

\def\EE{\mathbb{E}}
\def\PP{\mathbb{P}}

\def\aS3t{\overrightarrow{S_{t}}}

\DeclareMathOperator{\fst}{first}
\DeclareMathOperator{\lst}{last}
\DeclareMathOperator{\tail}{tail}
\DeclareMathOperator{\head}{head}

\def\onion{onion\xspace}
\newcommand{\os}[1]{${#1}$-onion-star\xspace}

\newcommand{\pqpair}{$(\cP,\cQ)$}

\def\bar#1{\overline{#1}}

\renewcommand{\leq}{\leqslant}
\renewcommand{\geq}{\geqslant}

\newcommand{\trimmed}[1]{\smash{\bar{#1}}}
\newcommand{\comp}[1]{{\langle #1\rangle\!}}

\newcommand{\N}{\mathbb{N}}

\title{On digraphs without onion star immersions\thanks{This work is a part of
    the projects CUTACOMBS (ŁB, OD, KO) and BOBR (MP) that have received funding from the European
    Research Council (ERC) under the European Union's Horizon 2020 research and
    innovation programme (grant agreements No 714704 and 948057, respectively).}}

\author[1]{\L{}ukasz Bo\.zyk}
\author[1,2]{Oscar Defrain}
\author[1,3]{\\Karolina Okrasa}
\author[1]{Micha\l{} Pilipczuk}
\affil[1]{Institute of Informatics, University of Warsaw, Poland}
\affil[2]{LIS, Aix-Marseille Universit\'e, France}
\affil[3]{Faculty of Mathematics and Information Science,\protect\\Warsaw University of Technology, Poland}
\date{November 29, 2022}

\begin{document}

\maketitle

\begin{textblock}{20}(-1.9, 7.38)
\includegraphics[width=40px]{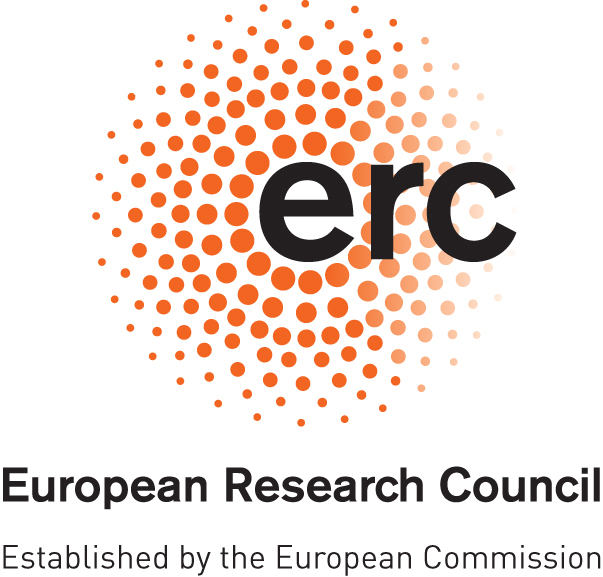}%
\end{textblock}
\begin{textblock}{20}(-2.15, 7.68)
\includegraphics[width=60px]{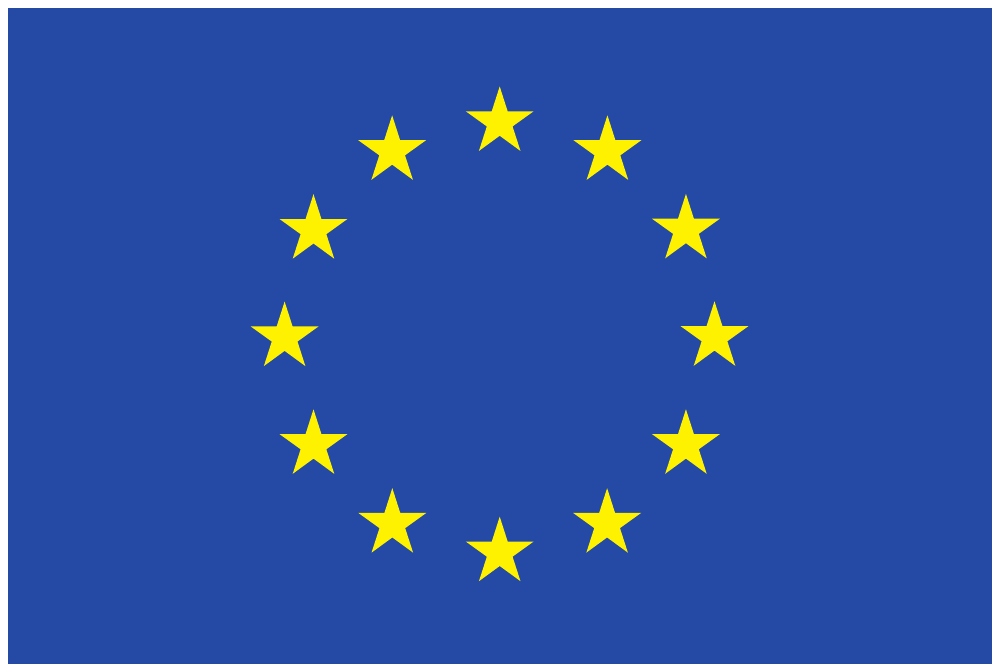}%
\end{textblock}

\begin{abstract}
The {\em{$t$-onion star}} is the digraph obtained from a star with $2t$ leaves by replacing every edge by a triple of arcs, where in $t$ triples we orient two arcs away from the center, and in the remaining $t$ triples we orient two arcs towards the center. Note that the $t$-onion star contains, as an immersion, every digraph on $t$ vertices where each vertex has outdegree at most $2$ and indegree at most $1$, or vice versa.

We investigate the structure in digraphs that exclude a fixed onion star as an immersion. The main discovery is that in such digraphs, for some duality statements true in the undirected setting we can prove their directed analogues. More specifically, we show the next two statements. 
\begin{itemize}[nosep]
 \item There is a function $f\colon \N\to \N$ satisfying the following: If a digraph $D$ contains a set $X$ of $2t+1$ vertices such that for any $x,y\in X$ there are $f(t)$ arc-disjoint paths from $x$ to $y$, then $D$ contains the $t$-onion star as an immersion.
 \item There is a function $g\colon \N\times \N\to \N$ satisfying the following: If $x$ and $y$ is a pair of vertices in a digraph $D$ such that there are at least $g(t,k)$ arc-disjoint paths from $x$ to $y$ and there are at least $g(t,k)$ arc-disjoint paths from $y$ to $x$, then either $D$ contains the $t$-onion star as an immersion, or there is a family of $2k$ pairwise arc-disjoint paths with $k$ paths from $x$ to $y$ and $k$ paths from $y$ to~$x$.
\end{itemize}
\end{abstract}

\section{Introduction}\label{sec:intro}
A graph $H$ is an {\em{immersion}} of a graph $G$ if one can injectively map vertices of $H$ to vertices of $G$ and edges of $H$ to edge-disjoint paths in $G$ so that the image of every edge connects the images of its two endpoints. Thus, immersions are an embedding notion for (undirected) graphs that is based on edge-disjointness and edge cuts, as opposed to the notions of minors and of topological minors, which are based on vertex-disjointness and (vertex) separations.

The Graph Minor series of Robertson and Seymour brought fundamental understanding of the minor order in graphs. The key components of this understanding are the Grid \mbox{Minor} Theorem~\cite{robertson1986graph}, which connects the existence of large grid minors with a dual notion of tree-likeness---the treewidth---and the Structure Theorem~\cite{RobertsonS03a}, which describes the structure in graphs that exclude a fixed graph $H$ as a minor. The analogues of these two tools have been understood also in the setting of immersions. In~\cite{wollan2015structure}, Wollan introduced the parameter {\em{tree-cut width}} and proved the Wall Immersion Theorem, a duality theorem connecting the tree-cut width to the largest size of a wall that can be found in a graph as an immersion. In the same work, he also gave a structure theorem that describes graphs excluding a fixed graph $H$ as an immersion. Also, as a part of the Graph Minors series, Robertson and Seymour~\cite{robertson2010graph} proved that the immersion order is a well quasi-order on graphs. We invite the reader to~\cite{DeVosMMS13,MarxW14,dvovrak2016structure,giannopoulou2016linear,GiannopoulouKRT21,Liu21,Liu22,BozykDOP22} for other works on structural properties of graphs with forbidden immersions.

The notion of an immersion can be naturally generalized to directed graphs (digraphs) by considering mapping every arc to a directed path leading from the image of the tail to the image of the head. Recently, there has been significant interest in proving directed analogues of the advances of the theory of (undirected) minors. In particular, Kawarabayashi and Kreutzer~\cite{KawarabayashiK15} proved the directed variant of the Grid Minor Theorem, while a series of recent papers~\cite{GiannopoulouKKK20,GiannopoulouKKK22} is gradually working towards a directed analogue of the Structure Theorem. From this point of view, it is natural to ask whether a similar structure theory can be developed for directed immersions. 
It seems that so far, not much is known in this direction. We remark that a meaningful structure theory for {\em{tournaments}} with forbidden immersions was developed, see~\cite{ChudnovskyS11,ChudnovskyFS12,Raymond18,FominP19,dmtcs:9276}, but this restricted setting is very different from the setting of general digraphs. Also, there has been work on finding appropriate degree restrictions that force the existence of immersions of large complete digraphs or large transitive tournaments, see~\cite{DeVosMMS12,Lochet19}.

The aim of this paper is to make the first modest steps towards a structure theory for digraphs excluding a fixed digraph as an immersion. A more specific motivation is to provide opening moves towards a statement linking directed wall immersions with suitable width measures for digraphs.

\paragraph*{Our contribution.} Consider the {\em{$t$-onion star}}: the digraph depicted in \Cref{fig:onion-star-intro} below.
It is obtained from the star with $2t$ leaves by replacing every edge with a triple of arcs; for $t$ edges two of the arcs are oriented towards the center and one away from the center, and for the remaining $t$ edges we use the reverse orientation. Note the following.

\begin{figure}[htb]
\centering{
\includegraphics[page=1, scale=1.25]{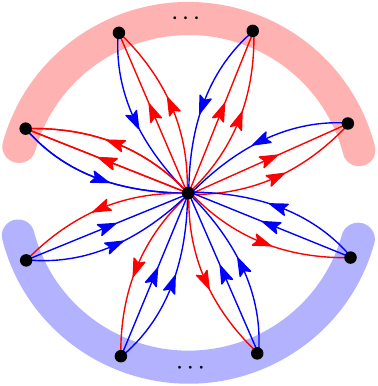}
}
\caption{The onion star.}
\label{fig:onion-star-intro}
\end{figure}

\begin{observation}\label{obs:rich-onion}
Suppose $D$ is a digraph on $t$ vertices where every vertex has outdegree at most~$2$ and indegree at most $1$, or vice versa. Then the $t$-onion star contains $D$ as an immersion.
\end{observation}

Therefore, excluding the $t$-onion star as an immersion, for some fixed $t\in \N$, is a weaker condition than excluding any fixed digraph $D$ satisfying the premise of \cref{obs:rich-onion}; this in particular applies to any reasonable notion of a directed wall. The $t$-onion star is a directed analogue of the graph $S_{3,t}$ that was used by Wollan in~\cite{wollan2015structure} as an obstruction commonly found in various avenues of his proof of the (undirected) Wall Immersion Theorem.

For two vertices $x,y$ in a digraph $D$, let $\mu(x,y)$ be the maximum number of arc-disjoint paths from $x$ to $y$ that one can find in $D$. As an opening step of his proof, Wollan proved the following statement.

\begin{theorem}[Wollan, {\cite[Lemma~1]{wollan2015structure}}]\label{lem:Wollan-no-cut}
 Suppose a graph $G$ contains a set $X$ consisting of $t+1$ vertices such that for all distinct $x,y\in X$, we have $\mu(x,y)\geq t^2$. Then $G$ contains the complete graph $K_t$ as an immersion.
\end{theorem}

The proof is a relatively easy application of flow-cut duality, which nonetheless crucially relies on the undirectedness of the graph. The main result of this paper is the following weak directed analogue of \cref{lem:Wollan-no-cut}.

\begin{theorem}\label{thm:no-cut}
There exists a function $f\colon \N\to \N$ such that the following holds.
Suppose a digraph $D$ contains a set $X$ consisting of $2t+1$ vertices such that for all distinct $x,y\in X$, we have $\mu(x,y)>f(t)$. Then $D$ contains the $t$-onion star as an immersion.
\end{theorem}

Note that in \cref{thm:no-cut} we obtain an obstruction that is weaker than that of \cref{lem:Wollan-no-cut}: the $t$-onion star instead of a complete digraph. 

The proof of \cref{thm:no-cut} applies the same basic flow-cut strategy as the proof of \cref{lem:Wollan-no-cut} due to Wollan, but there is a major issue. For any $x,y\in X$, the premise of \cref{thm:no-cut} provides the existence of a large family $\Pp_{x\to y}$ of arc-disjoint paths from $x$ to $y$, and of a large family $\Pp_{y\to x}$ of arc-disjoint paths from $y$ to $x$.
However,
in principle every path of $\Pp_{x\to y}$ could intersect every path of $\Pp_{y\to x}$, while for the abovementioned strategy to work, we need a family containing many arc-disjoint paths from $x$ to $y$ and many from $y$ to $x$ that are also arc-disjoint between each other. In general digraphs, there is no hope for exposing such a family: in~\Cref{fig:counterex} we give an example where $\mu(x,y)$ and $\mu(y,x)$ can be arbitrarily large, but one cannot find even two arc-disjoint paths: one from $x$ to $y$ and one from $y$ to $x$. However, we prove that the desired statement holds under the assumption of excluding an onion star.

\begin{figure}[H]
\centering{
\includegraphics[page=1, scale=1.25]{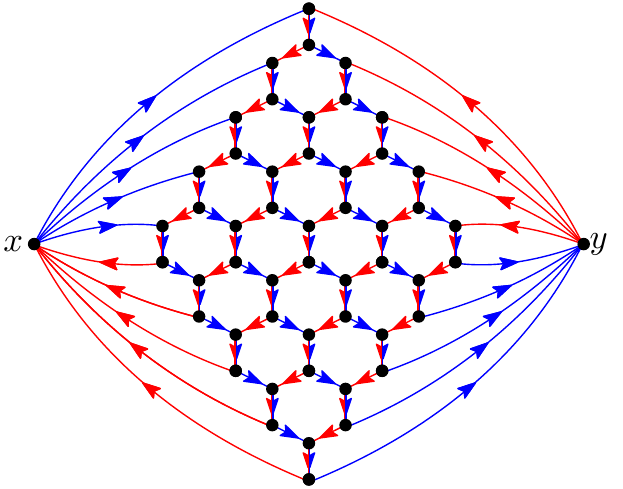}
}
\caption{An example of a digraph that contains a large family of arc-disjoint paths from $x$ to $y$ (in blue), a large family of arc-disjoint paths from $y$ to $x$ (in red), but no two arc disjoint paths such that one goes from $x$ to $y$, and the other one from $y$ to $x$.}
\label{fig:counterex}
\end{figure}

\begin{theorem}\label{thm:onion-harvesting}
There exists a function $g\colon \N\times \N\to \N$ such that the following holds.
 Suppose $D$ is a digraph and $x,y$ are two distinct vertices in $D$ such that $\mu(x,y)>g(t,k)$ and $\mu(y,x)>g(t,k)$, for some $t,k\in \N$. Then at least one of the following holds:
 \begin{itemize}[nosep]
  \item $D$ contains the $t$-onion star as an immersion;
  \item in $D$ there is a family of~$2k$ arc-disjoint paths consisting of $k$ paths from $x$ to $y$ and $k$ paths from $y$ to $x$.
 \end{itemize}
\end{theorem}

\cref{thm:onion-harvesting} is the main technical component in the proof of \cref{thm:no-cut}, and its proof spans most of this paper. We believe it is of independent interest, as in essence the statement shows  the following: under the assumption of excluding a fixed onion star, directed flows can be to some extent uncrossed so that we obtain basic properties observable in undirected flows.

\section{Preliminaries}\label{sec:prelim}
For a positive integer $n$, by $[n]$ we denote the set $\{1,\ldots,n\}$, and by $\mathbb{N}$ the set of positive integers.
If $f$ is a function whose domain and codomain are equal, and $t\geq 0$ is an integer, 
by $f^\comp{t}(x)$ we mean \[\underbrace{f(f(\ldots f}_{\text{$t$ times}}(x)\ldots)).\]

\paragraph{Digraphs and cuts.} 
For a directed graph $D$, by $V(D)$ and $A(D)$ we denote the vertex and the arc set of $D$, respectively.
We allow the existence of parallel arcs (multiple arcs with the same pair of endpoints), but we do not allow loops (arcs with both endpoints at the same vertex). 
For $xy \in A(D)$, we call $x$ and $y$ the \emph{tail} and the \emph{head} of $xy$, respectively. For an arc $a\in A(D)$ we denote by $\tail(a)$ and $\head(a)$ its tail and head, respectively.
For a set $X \subseteq V(D)$, by $D[X]$ we denote the subgraph induced by $X$. Moreover we define 
\begin{align*}
\delta^+(X)&\coloneqq \{a\in A(D) : \tail(a)\in X\text{ and }\head(a)\notin X\},\\
\delta^-(X)&\coloneqq \{a\in A(D) : \tail(a)\notin X\text{ and }\head(a)\in X\}.
\end{align*}
If $X=\{x\}$, we omit internal brackets and write $\delta^+(x)$ and $\delta^-(x)$, respectively.

A \emph{path} in a digraph $D$ is a sequence $(a_1, a_2, \ldots, a_k)$ of pairwise distinct arcs of $D$ with the property that $\head(a_i)=\tail(a_{i+1})$ for every $i\in[k-1]$.
Given a path $P=(a_1,a_2,\ldots,a_k)$ we put $\fst(P)\coloneqq a_1$, $\lst(P)\coloneqq a_k$, and $A(P)\coloneqq\{a_1,a_2,\ldots,a_k\}$. Path $P$ is \emph{simple} if $\head(a_i)=\tail(a_j)$ if and only if $j=i+1$, i.e., if the tail of the first arc and the heads of all arcs form $k+1$ distinct vertices. Let $<_P$ be the natural linear order of $A(P)$ defined by $a_i<_P a_j$ if and only if $i<j$. We define \emph{trimmed paths}:
\begin{align*}
P(\to a_i)&\coloneqq (a_1,a_2,\ldots,a_{i-1}),\\
P(\to a_i]&\coloneqq (a_1,a_2,\ldots,a_i),\\
P(a_i\to)&\coloneqq (a_{i+1},a_{i+2},\ldots,a_k),\\
P[a_i\to)&\coloneqq (a_i, a_{i+1},\ldots, a_k),\\
P(a_i,a_j)&\coloneqq (a_{i+1},a_{i+2},\ldots,a_{j-1}).
\end{align*}
If $s=\tail(a_1)$ and $t=\head(a_k)$, then $P$ is called an \emph{$(s,t)$-path}. If $\cP$ is a family of paths and $Q$ is a path, then we put $A(\cP)\coloneqq \bigcup_{P\in\cP} A(P)$ and $\cP(Q)\coloneqq\{P\in \cP\mid A(P)\cap A(Q)\neq\emptyset\}$.
For two arc-disjoint paths $P=(p_1,\ldots,p_n)$ and $Q=(q_1,\ldots,q_m)$ such that $\head(p_n)=\tail(q_1)$, we denote by $PQ$ the \emph{concatenation} of $P$ and $Q$, i.e., the path $(p_1,\ldots,p_n,q_1,\ldots,q_m)$.

A \emph{cut} in $D$ is a partition $(A,B)$ of $V(D)$. 
An \emph{$(a,b)$-cut} in $D$, where $a,b\in V(D)$, is a cut $(A,B)$ with $a\in A$ and $b\in B$.
The \emph{size} of the cut $(A,B)$ is $|\delta^+(A)|=|\delta^-(B)|$. 

The classical theorem of Menger describes the relation between the size of $(a,b)$-cuts and the number of arc-disjoint paths connecting $a$ and $b$. 

\begin{theorem}[Menger~\cite{menger1927allgemeinen}]\label{thm:menger}
Let $D$ be a digraph and let $a,b \in V(D)$.
The maximum number of arc-disjoint $(a,b)$-paths equals the minimum size of an $(a,b)$-cut.
\end{theorem}

If $\cP$ and $\cQ$ are two families of paths in a digraph $D$, then by \emph{intersection graph} of the pair \pqpair{} we mean the undirected bipartite graph with bipartition classes $\cP$, $\cQ$, and an edge between $P \in \cP$ and $Q \in \cQ$ if $P$ and $Q$ have at least one common arc.

\paragraph{Immersions.}
We say that a digraph $D$ admits a digraph $H$ as an \emph{immersion} if there exists an
\emph{immersion model} of $H$ in $D$, that is, a mapping $\pi$ defined on the vertices and arcs of $H$ as follows:
\begin{itemize}[nosep]
\item $\pi$ maps vertices of $H$ to pairwise different vertices of $D$;
\item $\pi$ maps each arc $xy \in A(H)$ to a path in $D$ with tail $\pi(x)$ and head $\pi(y)$;
\item paths in $\{\pi(a) : a \in A(H)\}$ are pairwise arc-disjoint.
\end{itemize}
It is clear that $D$ admits $H$ as an immersion if and only if there exists an immersion model of $H$ in $D$ in which all arcs are mapped to simple paths.

Let $t\geq 1$.
An \emph{onion}\footnote{We note that onions are directed variants of pumpkins, which were studied e.g.~in~\cite{joret2014hitting}.} is the digraph $\smash{\overrightarrow{O}}$ which consists of two vertices $x$ and $y$, two arcs $yx$ and one arc $xy$. 
The vertices $x$ and $y$ are called \emph{roots} of the onion, with $x$ being also called the \emph{sink} of the onion, and $y$ being the \emph{source} of the onion. If $\pi$ is an immersion model of $\smash{\overrightarrow{O}}$ in $D$, then the vertices $\pi(x)$, $\pi(y)$ of $D$ will be called roots (accordingly sink and source) of this immersion.
A \emph{\os{t}} is the digraph $\smash{\overrightarrow{S_{t}}}$ with the set of vertices $\{x,y_1,\ldots,y_t,z_1,\ldots,z_{t}\}$ and set of arcs $\bigcup_{i=1}^t A_i$, where each $A_i$ consists of single arcs $y_ix$ and $xz_{i}$, and double arcs $xy_i$ and $z_{i}x$ (see \Cref{fig:onion}).

\begin{figure}[h]
\centering
\includegraphics[page=2,scale=1.25]{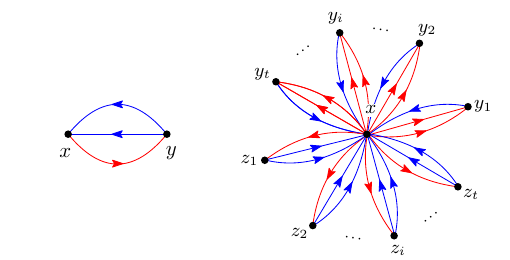}
\caption{An onion (left) with sink $x$ and source $y$ and a \os{t} (right).
Blue and red arcs are incoming and outgoing arcs of $x$, respectively.}
\label{fig:onion}
\end{figure}

Observe that the graph $\smash{\overrightarrow{S_{t}}[\{y_1,\ldots,y_t,z_1,\ldots,z_t\}]}$ is arcless and for every $i=1,2,\ldots,t$, the graphs $\smash{\overrightarrow{S_{t}}[\{x,y_i\}]}$ and $\smash{\overrightarrow{S_{t}}[\{x,z_i\}]}$ are onions. Moreover, $x$ is the source of $\smash{\overrightarrow{S_{t}}[\{x,y_i\}]}$ and the sink of $\overrightarrow{S_{t}}[\{x,z_i\}]$. 

\paragraph{Graphs.} 
If $G$ is a simple undirected graph, we denote by $V(G)$ and $E(G)$ its sets of vertices and edges, respectively.
For a vertex $v \in V(G)$, let $\deg_G(v)$ be the number of neighbors of $v$ in $G$.
We say that two disjoint sets $A, B \subseteq V(G)$ are \emph{complete} to each other if every $a \in A$ is adjacent to every $b \in B$.
Similarly, $A$ and $B$ are \emph{anti-complete} to each other if every $a \in A$ is nonadjacent to every $b \in B$.

An undirected graph $G$ is \emph{bipartite} if there exists a partition of $V(G)$ into sets $X$ and $Y$ such that each edge of $G$ has one endpoint in $X$ and another in $Y$. 
A bipartite graph is \emph{balanced} if there exists a choice of the bipartition classes of $G$ such that the numbers of vertices in the two classes are equal.
By $K_{n,n}$ we denote the complete balanced graph on $2n$ vertices.
Let $n$ be a positive integer. 
By $b(n)$ we denote the minimum integer such that in every balanced graph on $2b(n)$ vertices, with bipartition classes $X$ and $Y$, either there exists an induced subgraph isomorphic to $K_{n,n}$, or there exist two sets $A \subseteq X, B\subseteq Y$, each of size $n$, that are anti-complete to each other. 
The existence of $b(n)$ for each $n$ follows from \cite{DBLP:journals/ejc/Thomason82}.
\begin{theorem}[Thomason~\cite{DBLP:journals/ejc/Thomason82}]\label{thm:thomasen}
For every $n \geq 1$ we have $b(n)\leq 2^n(n-1) + 1.$
\end{theorem}

In other words, if a bipartite graph $G$ has sufficiently many vertices then we can always find in it two sets of size $n$, each contained in a different bipartition class, that are either complete or anti-complete to each other.
In the same flavor, the following classical result of Kővári, Sós and Turán~\cite{kHovari1954problem} gives a lower bound on the number of edges of a (not necessarily bipartite) graph $G$ so that it contains a large complete balanced subgraph. 
\begin{theorem}[Kővári, Sós, Turán~\cite{kHovari1954problem}]\label{thm:kst}
There exists a function $c\colon \N\to (0,\infty)$ such that if an $n$-vertex graph $G$ has at least $c(k) n^{2-1/k}$ edges for some $k \in \mathbb{N}$, then it contains $K_{k,k}$ as a subgraph.
\end{theorem}

\section{Onion Harvesting Lemma}\label{sec:wc-lemma}

This section is devoted to the proof of \cref{lem:wc}, the main conceptual piece of this work.

Let $\cP$ be a family of paths in a digraph $D$ and let $X \subseteq V(D)$. 
We say that $\cP$ \emph{starts at $X$} if for every $P \in \cP$ it holds that $\tail(\fst(P))\in X$, and that it \emph{ends at $X$} if for every $P \in \cP$ it holds that $\head(\lst(P))\in X$. If $X=\{x\}$, we omit the brackets and say, respectively, that $\cP$ starts or ends at $x$.

Consider a pair \pqpair{} where $\cP$ and $\cQ$ are families of pairwise arc-disjoint simple paths in $D$ such that there exists $x \in V(D)$ with $\cP$ starting at $x$ and $\cQ$ ending at $x$.
We say that \pqpair{} is a \emph{well-crossing pair rooted at $x$} (or simply \emph{well-crossing pair}) if for every $P \in \cP$ and $Q \in \cQ$ we have $A(P) \cap A(Q) \neq \emptyset$, i.e., if the intersection graph of $(\cP,\cQ)$ is complete. 

We shall prove that, for every $t\geq 1$, if a digraph $D$ contains a well-crossing pair $(\cP,\cQ)$ with both families $\cP$ and $\cQ$ sufficiently large (in terms of $t$), then $D$ admits a $t$-onion-star as an immersion.
This is formalized by the following lemma.

\begin{lemma}[Onion Harvesting Lemma]\label{lem:wc}
Let $t\geq 1$ be an integer and $D$ be a digraph. There exists a function $F\colon\N\to \N$ with the following property: if in $D$ there exists a well-crossing pair $(\cP,\cQ)$ with $|\cP|\geq F(t)$ and $|\cQ|\geq F(t)$, then in $D$ there exists an immersion model of a $t$-onion-star.
\end{lemma}

Let $(\cP,\cQ)$ be a well-crossing pair. 
In the following, we call \emph{crossing} of $(\cP,\cQ)$ an arc $e$ that belongs to $A(P) \cap A(Q)$ for some $P \in \cP$ and $Q \in \cQ$, and specify \emph{$P$-crossing} (resp.~\emph{$Q$-crossing}, \emph{$(P,Q)$-crossing}) for a crossing of $(\cP,\cQ)$ contained in $P$ (resp.~$Q$, $A(P)\cap A(Q)$).
A crossing $e$ is \emph{$(P,Q)$-safe} (or simply \emph{safe}) if there exist at least $|\cQ|/3$ paths in $\cQ \setminus \{Q\}$ whose $<_P$-minimal crossing with $P$ precedes $e$ in $<_P$.
Otherwise, it is said to be $(P,Q)$-\emph{dangerous} (or simply \emph{dangerous}).
We say that a path $P$ is \emph{$Q$-dangerous} if it contains a dangerous $(P,Q)$-crossing.
If a path $Q\in \cQ$ is trimmed to $\trimmed{Q}$, then we say that $P$ is \emph{$\trimmed{Q}$-dangerous} if there exists a dangerous $(P,Q)$-crossing belonging to $\trimmed{Q}$. 

In order to prove \Cref{lem:wc}, we will repeatedly ``harvest'' immersions of single onions rooted at the root of $(\cP,\cQ)$, keeping appropriately large well-crossing pairs disjoint with previously found onions to enable gaining new ones.

\begin{lemma}\label{lem:single-out}
Let $n\geq 1$ be an integer. There exists a function $f\colon\N\to\N$ with the following property: If in a digraph $D$ there exists a well-crossing pair $(\cP,\cQ)$ rooted at $x$ with $|\cP|= f(n)$ and $|\cQ|= f(n)$, then in $D$ there exist an immersion model $\pi$ of an onion with source $x$, and a well-crossing pair $(\cP^*,\cQ^*)$ rooted at $x$ with $|\cP^*|= n$, $|\cQ^*|= n$ such that:
\begin{itemize}[nosep]
\item $A(\cP^*) \subseteq A(\cP)$, and $A(\cQ^*) \subseteq A(\cQ)$;
\item all paths in $\cP^*\cup\cQ^*$ are arc-disjoint with $\pi$;
\item no arc in $A(\cP^*)\cap A(\cQ^*)$ has the sink of $\pi$ for tail.
\end{itemize}
\end{lemma}

\begin{proof}[Proof of~\cref{lem:single-out}]
Let $c$ be the function from \Cref{thm:kst}. Define functions $g$ and $f$ by:
\[g(n)=\max\left\{8,\left\lceil\frac12(16c(2n))^{2n}\right\rceil\right\},\qquad f(n)=\max\left\{6n,3\cdot\left\lceil\frac14(36c(g(n)))^{g(n)}\right\rceil\right\}.\]
Note that since the values of $f$ are divisible by $3$, so are the sizes of $\cP$ and $\cQ$ for $(\cP,\cQ)$ a well-crossing pair satisfying the assumptions of the lemma.

In the following, we consider one such pair $(\cP,\cQ)$ rooted at a vertex $x$ and fix $Q \in \cQ$. 
Let $e$ be the $<_Q$-smallest arc of $Q$ such that $|\cP(Q(e\to))| = |\cP|/3$. 
Observe that $e$ is a crossing satisfying $|\cP(Q[e\to))|=|\cP(Q(e\to))|+1$. 
Let $\trimmed{Q}:=Q[e\to)$.
We distinguish two cases depending on the number of $\trimmed{Q}$-dangerous paths in $\cP(\trimmed{Q})$. These cases are depicted in \cref{fig:extracted-onion} and \cref{fig:extracted-onion2}, respectively.

\textbf{Case 1.} Suppose that $\cP(\trimmed{Q})$ contains at least two $\trimmed{Q}$-dangerous paths. 
Let $d$ be the $<_Q$-greatest dangerous crossing in $Q$ and let $P_2$ be the path such that $d$ is a $(P_2,Q)$-crossing. 
Then $d$ belongs to $\trimmed{Q}$.
Let $d'$ be the $<_{P_2}$-smallest $(P_2,Q)$-crossing belonging to $\trimmed{Q}(e'\to)$. 
Clearly since $d$ is dangerous and belongs to $\trimmed{Q}$, $d'$ must be dangerous (with possibly $d=d'$). 
Let $e'$ be the $<_Q$-greatest dangerous crossing in $A(Q)\setminus A(P_2)$ and let $P_1\in\cP(\trimmed{Q})\setminus\{P_2\}$ be the path such that $e'$ is a $(P_1,Q)$-crossing (with possibly $e=e'$).

\begin{figure}[H]
    \centering
    \includegraphics[scale=1.25, page=1]{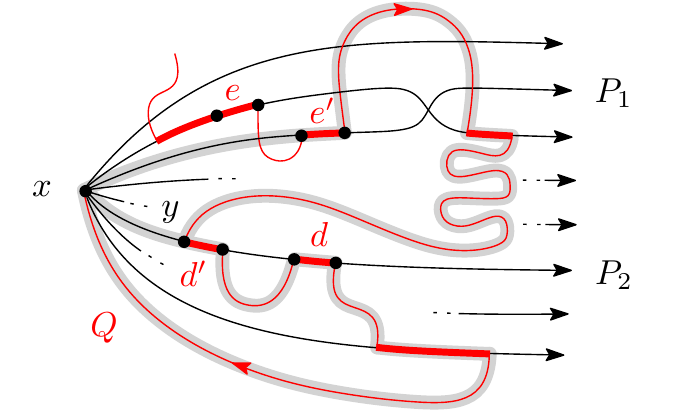}
    \caption{The situation of Case 1 when $e\neq e'$ and $d\neq d'$. Paths of $\cP$ are depicted in black, $Q$ is in red and $Q$-crossings are in thick red. An onion immersion model with source $x$ and sink $y$ is underlined in gray.}\label{fig:extracted-onion}
\end{figure}

Let $y=\tail(d')$.
We claim that the $(x,y)$-paths $P_1(\to e']Q(e',d')$, $P_2(\to d')$ and the $(y,x)$-path $Q[d'\to)$ form an immersion model $\pi$ of an \onion with source $x$ and sink $y$. 
Observe that $P_1(\to e']$ is arc-disjoint with $Q(e'\to)$ as $e'$ is the $<_Q$-greatest dangerous $(P_1,Q)$-crossing, and having a safe $(P_1,Q)$-crossing in $A(P_1(\to e']) \cap A(Q(e' \to ))$ would imply that $e'$ is also safe, a contradiction.
Similarly, the $<_{P_2}$-minimality of $d'$ asserts that $P_2(\to d')$ is arc-disjoint with $Q[e' \to)$.
Note that these arguments hold even if $e=e'$ and $d=d'$.
Finally, $Q(e',d')$ and ${Q[d'\to)}$ are arc-disjoint as $Q$ is simple.
Hence, the path ${P_1(\to e']Q(e',d')}$ is a well-defined (although not necessarily simple) $(x,y)$-path, and paths $P_1(\to e']Q(e',d)$, $P_2(\to d')$, and $Q[d'\to)$ are pairwise arc-disjoint, yielding an immersion model $\pi$ of an onion with source $x$ in $D$.

Let $\cQ^\circ \subseteq \cQ$ be the set of paths which are arc-disjoint both with $P_1(\to e']$ and $P_2(\to d']$ (so in particular $Q\notin \cQ^\circ$).
Since $e'$ and $d'$ are both dangerous, $|\cQ^\circ|> |\cQ|-|\cQ|/3-|\cQ|/3=|\cQ|/3\geq 2n$ by the definition of $f$.
Let us pick an arbitrary collection $\cP^*\subseteq\cP\setminus \{P_1,P_2\}$ of paths not intersecting $Q[e'\rightarrow)$ with $|\cP^*|=n$.
Such a collection exists as $|\cP(Q(e\to))| = |\cP|/3$ yielding $\smash{\frac23}|\cP|\geq 4n$ candidate paths for $\cP^*$.
Note that each path in such a collection has at most one arc whose tail is $y$ (as the paths are simple), so there are at most $|\cP^*|=n$ arcs in $A(\cQ^\circ)\cap A(\cP^*)$ whose tail is $y$. Let $\cQ^*\subseteq \cQ^\circ$ comprise of any $n$ paths disjoint with these arcs.
Clearly, since $\cP^* \subseteq \cP$ and $\cQ^* \subseteq \cQ$,  $(\cP^*,\cQ^*)$ is a well-crossing pair. Hence $\pi$, $\cP^*$ and $\cQ^*$ satisfy all the desired properties.

\textbf{Case 2.} Suppose that there is at most one $\trimmed{Q}$-dangerous path in $\cP(\trimmed{Q})$. 
Let $\cS$ be the set of those paths in $\cP(\trimmed{Q})$ whose all crossings with $\trimmed{Q}$ are safe, and let
\[
    \trimmed{\cS}\coloneqq\left\{P(\to\min{_{<_P}}\{A(P)\cap A(\trimmed{Q})\}) \mid P\in \cS\right\}.
\]
This is the set of paths from $\cS$ trimmed so that they start at $x$,  have no arcs in common with $\trimmed{Q}$ and are maximally long.
Since $|\cP(\trimmed{Q})| = |\cP|/3+1$ and there is at most one dangerous path in $\cP(\trimmed{Q})$, we derive $|\trimmed{\cS}|=|\cS|\geq|\cP(\trimmed{Q})|-1=|\cP|/3$. 
We shall first exhibit large well-crossing subfamilies of $\trimmed{\cS}$ and $\cQ$.

\begin{figure}[H]
    \centering
    \includegraphics[scale=1.25, page=2]{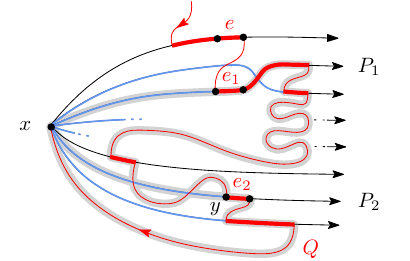}
    \caption{The situation of Case 2. Paths in $\trimmed{\cS}$ are depicted in blue, $Q$ is in red and $Q$-crossings are in thick red. Paths $P_1$ and $P_2$ are those of \Cref{cla:paths}. An onion immersion model with source $x$ and sink $y$ is underlined in gray.}\label{fig:extracted-onion2}
\end{figure}

\begin{claim}\label{cla:same-paths}
There exist subfamilies $\trimmed{\cS}'\subseteq \trimmed{\cS}$, $\cQ'\subseteq \cQ\setminus \{Q\}$ such that $|\trimmed{\cS}'|=g(n)$, $|\cQ'|=g(n)$ and $(\trimmed{\cS}',\cQ')$ is well-crossing.
\end{claim}

\begin{proof}[Proof of claim]
Let $G$ be the intersection graph of $(\trimmed{\cS},\cQ)$. 
By definition, each trimmed path $\trimmed{P}$ in $\trimmed{\cS}$ originates from a path $P$ in $\cS$ whose $(P,Q)$-crossings $e$ within $\trimmed{Q}$ are safe, i.e., there are at least $|\cQ|/3$ distinct paths among $\cQ\setminus \{Q\}$ crossing $\trimmed{P}$ before $e$ with respect to $<_{\trimmed{P}}$.
Since moreover $|\trimmed{\cS}|\geq|\cP|/3$, we get
\[
    |E(G)|=\sum_{\trimmed{P}\in\trimmed{\cS}}\deg_G(\trimmed{P})\geq \sum_{\trimmed{P}\in\trimmed{\cS}}\frac{|\cQ|}{3}\geq \frac{|\cP|}{3} \cdot  \frac{|\cQ|}{3}=\left(\frac{f(n)}{3}\right)^2.
\]
Note in addition that $\frac43 f(n)=\frac{|\cP|}{3} + |\cQ| \leq |V(G)|\leq 2f(n)$. 
By \Cref{thm:kst} there exists a constant $c_1=c(g(n))$ such that if 
\[
    |E(G)|\geq c_1\cdot |V(G)|^{2-\frac1{g(n)}},
\]
then $G$ contains $K_{g(n),g(n)}$ as a subgraph, corresponding to the desired well-crossing pair $(\trimmed{\cS}',\cQ')$. 
For the above inequality to hold it is enough to guarantee that
\begin{equation}\label{eq:kstapp1}
\frac{|E(G)|}{|V(G)|^2}\geq \frac{c_1}{|V(G)|^{\frac{1}{g(n)}}}.
\end{equation}
By $|V(G)|\leq 2f(n)$ and $|E(G)|\geq \frac19 f(n)^2$ we first get
\begin{equation}\label{eq:kstapp1step1}
\frac{|E(G)|}{|V(G)|^2}\geq \frac{\frac19 f(n)^2}{4f(n)^2}=\frac{1}{36}.
\end{equation}
From the definition of $f$ it follows that $f(n)\geq \frac{3}{4}(36c_1)^{g(n)}$, so equivalently
\begin{equation}\label{eq:kstapp1step2}
\frac{1}{36}\geq \frac{c_1}{\left(\frac43 f(n)\right)^{\frac1{g(n)}}}.
\end{equation}
Moreover, $|V(G)|\geq\frac43 f(n)$ gives
\begin{equation}\label{eq:kstapp1step3}
\frac{c_1}{\left(\frac43 f(n)\right)^{\frac1{g(n)}}}\geq \frac{c_1}{|V(G)|^{\frac{1}{g(n)}}}.
\end{equation}
Putting together inequalities \cref{eq:kstapp1step1},  \cref{eq:kstapp1step2},  \cref{eq:kstapp1step3}, we obtain \cref{eq:kstapp1}, which finishes the proof.
\cqed
\end{proof}

Consider the pair $(\trimmed{\cS}',\cQ')$ given by \cref{cla:same-paths} and
let $\cS'$ be the set of paths in $\cS$ from which $\trimmed{\cS}'$ originates (i.e., the paths from $\trimmed{\cS}'$ before trimming).
For two distinct paths $P_1,P_2 \in \cS'$, we denote  
\[
    \trimmed{\cQ}'(P_1,P_2)\coloneqq\left\{R(\max{_{<_R}}\{A(\{P_1,P_2\})\cap A(R)\}\to)\ \mid\ R \in \cQ'\right\}.
\]
This is the set of paths from $\cQ'$ trimmed so that they end at $x$, have no arcs in common with $P_1$ or $P_2$ and are longest possible.

\begin{claim}\label{cla:paths}
There exist distinct $P_1,P_2 \in \cS'$ and subfamilies $\cP^\circ\subseteq \trimmed{\cS}'$, $\cQ^\circ\subseteq \trimmed{\cQ}'(P_1,P_2)$ such that $|\cP^\circ| = 2n$, $|\cQ^\circ| = 2n$  and $(\cP^\circ,\cQ^\circ)$ is well-crossing.
\end{claim}

\begin{proof}[Proof of claim]
Let $k\coloneqq g(n)=|\cQ'|=|\trimmed{\cS}'|=|\cS'| $.
Choose a pair $\{P_1,P_2\}\subseteq \cS'$ uniformly at random (from the set of all $\binom{k}{2}$ two-element subsets). Let $\trimmed{\cQ}'\coloneqq\trimmed{\cQ}'(P_1,P_2)$ and let $G$ be the intersection graph of $(\trimmed{\cS}',\trimmed{\cQ}')$. Clearly, $|\trimmed{\cQ}'|=|\cQ'|=k$. For $\trimmed{R}\in \trimmed{\cQ}'$ let $N_{\trimmed{R}}$ be the random variable denoting the number of paths $P\in \trimmed{\cS}'$ such that $A(\trimmed{R})\cap A(P)\neq\emptyset$, i.e., the degree of $\trimmed{R}$ in $G$.

Fix $\trimmed{R}\in \trimmed{\cQ}'$ and let $R\in \cQ'$ be the path from which $\trimmed{R}$ originates (i.e., the untrimmed counterpart of $\trimmed{R}$). 
For every $P\in \trimmed{\cS}'$ let $a(P)\coloneqq \max_{<_{R}}\{A(P)\cap A(R)\}$ be the $<_{R}$-greatest common arc of $P$ and $R$. 
As $(\trimmed{\cS}',\cQ')$ is well-crossing, the $k$-element set $\{a(P)\mid P\in\trimmed{\cS}'\}\subseteq A(R)$ is naturally linearly ordered by $<_{R}$.
Label all paths from $\trimmed{\cS}'$ with numbers from $1$ to $k$ accordingly with the order of appearance of arcs $a(P)$ on $R$, and let $\ell(P)$ be the label of $P$.
Note that $N_{\trimmed{R}} \geq k-\max(\ell(P_1),\ell(P_2))$ and for every $i=0,1,\ldots,k-1$,
\[
    \PP(\max(\ell(P_1),\ell(P_2))=i+1)=\frac{i}{\binom{k}{2}}.
\]
By a straightforward computation we obtain
\[
    \EE N_{\trimmed{R}} =\sum_{i=0}^{k-1}\PP(N_{\trimmed{R}}=k-i-1)(k-i-1)=\frac1{\binom{k}{2}}\sum_{i=0}^{k-1}i(k-i-1)=\frac{k-2}{3}.
\]

Note that the distribution of $N_{\trimmed{R}}$ depends only on labels of $P_1$ and $P_2$ and since these are chosen randomly, for different $\trimmed{R},\trimmed{S}\in\trimmed{\cQ}'$ the variables $N_{\trimmed{R}}$, $N_{\trimmed{S}}$ are identically distributed.  
By the above observations and since $k=g(n)\geq 8$, we have
\[\EE |E(G)|\geq \EE \sum_{\trimmed{S}\in\trimmed{\cQ}'} N_{\trimmed{S}} \geq k\cdot\frac{k-2}{3}\geq \left(\frac{k}{2}\right)^2.\]
It follows that there exists a choice of $P_1$, $P_2$ such that
\begin{equation}\label{kstapp2s1}
|E(G)|\geq \left(\frac{g(n)}{2}\right)^2.
\end{equation}
Similarly as in the proof of \Cref{cla:same-paths}, we observe by \Cref{thm:kst} that there exists a constant $c_2=c(2n)$ such that if 
\[|E(G)|\geq c_2\cdot |V(G)|^{2-\frac1{2n}},\]
then $G$ contains $K_{2n,2n}$ as a subgraph, corresponding to the desired well-crossing pair $(\cP^\circ,\cQ^\circ)$. For the above inequality to hold it is enough to guarantee that
\[\frac{|E(G)|}{|V(G)|^2}\geq \frac{c_2}{|V(G)|^{\frac1{2n}}}.\]
This is obtained observing that
\[\frac{|E(G)|}{|V(G)|^2}=\frac{|E(G)|}{4(g(n))^2}\geq \frac1{16}\geq \frac{c_2}{(2g(n))^{\frac1{2n}}}=\frac{c_2}{|V(G)|^{\frac1{2n}}},\]
where the first inequality follows from \cref{kstapp2s1}, and the second inequality follows from the definition of $g$ ensuring that $g(n)\geq \frac12(16c_2)^{2n}$.
\cqed
\end{proof}

Consider $P_1$, $P_2$ and $(\cP^\circ,\cQ^\circ)$ as in the claim above and denote
$e_i=\min_{<_{P_i}}\{A(\trimmed{Q})\cap A(P_i)\}$ for $i=1,2$. Assume without loss of generality that
$e_1<_Q e_2$, i.e., that the $<_{P_1}$-smallest $(P_1,\trimmed{Q})$-crossing is further from $x$ on $\trimmed{Q}$ than the $<_{P_2}$-smallest $(P_2,\trimmed{Q})$-crossing. Let $y=\tail(e_2)$.
By the choice of $e_1$ and $e_2$, the $(x,y)$-paths $P_2(\to e_2)$, $P_1(\to e_1]Q(e_1,e_2)$ and the $(y,x)$-path $Q[e_2\to)$ are arc-disjoint and hence form an immersion model $\pi$ of an onion with source $x$ and sink $y$. 
Moreover, this model is disjoint with $A(\cP^\circ\cup \cQ^\circ)$ as $\cQ^\circ\subseteq \trimmed{\cQ}'$ with $\trimmed{\cQ}'$ arc-disjoint with both $P_1$ and $P_2$. 

The last step (to ensure that $y$ will not be the tail of a crossing in the found well-crossing pair) is performed similarly as in the first case: we pick an arbitrary collection $\cP^\ast$ of exactly $n$ paths among the $2n$ paths in $\cP^\circ$, and select a subfamily $\cQ^\ast$ consisting of $n$ paths which do not contain crossings of tail $y$ among the $2n$ paths in $\cQ^\circ$.

That concludes the proof of \Cref{lem:single-out}.
\end{proof}

Now, we observe that by iterating \Cref{lem:single-out} $t$ times, we can obtain an immersion model of a digraph consisting of $t$ onions with a common source.

\begin{corollary}\label{cor:multiple-out}
Let $n$, $t$ be positive integers. Suppose that in a digraph $D$ there exists a well-crossing pair $(\cP,\cQ)$ rooted at $x$ with $|\cP|=|\cQ|= f^\comp{t}(n)$,
where $f$ is the function from \Cref{lem:single-out}.
Then in $D$ there exists a family $\cH$ of $t$ arc-disjoint immersion models of onions with common source $x$, whose all sinks are pairwise different, and a well-crossing pair $(\cP^*,\cQ^*)$ rooted at $x$ with $|\cP^*|= n$, $|\cQ^*|= n$ such that all paths in $\cP^*\cup\cQ^*$ are arc-disjoint with every element of $\cH$.
\end{corollary}

\begin{proof}
Put $\cP_0\coloneqq \cP$ and $\cQ_0\coloneqq \cQ$. For $i=0,1,\ldots,t-1$ let us apply \Cref{lem:single-out} to the pair $(\cP_i,\cQ_i)$ to get an immersion model of an onion $\pi_{i+1}$ with sink $y_{i+1}$, and a~well-crossing pair $(\cP_{i+1},\cQ_{i+1})$. 

Now it is enough to take $\cP^*=\cP_t$, $\cQ^*=\cQ_t$ and $\cH=\{\pi_i\}^{t}_{i=1}$.
Indeed, to see that the sinks $y_1,\ldots,y_t$ are distinct, recall that for every $i \in \{0,\ldots, t-1\}$ no arc in $A(\cP_{i+1})\cap A(\cQ_{i+1})$ has $y_{i+1}$ as a tail.
Since $A(\cP_i) \subseteq A(\cP_j)$, and $A(\cQ_i) \subseteq A(\cQ_j)$ for every $0\leq j \leq i$, vertices $y_1,\ldots,y_i$ do not appear as tails of arcs in $A(\cP_{i}) \cap A(\cQ_{i})$, hence $y_{i+1}$ must be distinct from $y_1,\ldots,y_i$.
The remaining properties of the immersion model are straightforward to verify. 
\end{proof}

An analogous lemma can be designed to enable finding onions whose sink (not source) is~$x$. The proof basically follows from reversing all arcs in the arguments delivered previously.

\begin{corollary}\label{cor:multiple-in}
Let $n$, $t$ be positive integers. Suppose that in a digraph $D$ there exists a well-crossing pair $(\cP,\cQ)$ rooted at $x$ with $|\cP|=|\cQ|= f^\comp{t}(n)$,
where $f$ is the function from \Cref{lem:single-out}.
Then in $D$ there exists a family $\cH$ of $t$ arc-disjoint immersion models of onions with common \underline{sink}~$x$, whose all \underline{sources} are pairwise different, and a well-crossing pair $(\cP^*,\cQ^*)$ rooted at $x$ with $|\cP^*|= n$, $|\cQ^*|= n$ such that all paths in $\cP^*\cup\cQ^*$ are arc-disjoint with every element of $\cH$.
\end{corollary}

\begin{proof}
It is enough to apply \Cref{cor:multiple-out} to the digraph $\overleftarrow{D}$, in which each arc of $D$ is replaced with a reversed arc, and to the (swapped) well-crossing pair $\smash{(\overleftarrow{\cQ},\overleftarrow{\cP})}$ consisting of families of reversed paths from $\cQ$ and~$\cP$, respectively. 
\end{proof}

Finally, we use \cref{cor:multiple-out} and \cref{cor:multiple-in} to prove \Cref{lem:wc}.

\begin{proof}[Proof of \Cref{lem:wc}]
Take $F(t)=f^\comp{4t}(1)$, where $f$ is the function satisfying \Cref{lem:single-out}. Let $x$ be the root of the pair $(\cP,\cQ)$. Applying \Cref{cor:multiple-out} to $D$, $(\cP,\cQ)$, $n=f^\comp{2t}(1)$ and $2t$ (in place of~$t$) we find a family $\cH^+=\{\pi^+_i\}_{i=1}^{2t}$ of $2t$ arc-disjoint immersion models of onions with source $x$ and mutually different sinks $\{y_i\}_{i=1}^{2t}$, along with a well-crossing pair $(\cP^*,\cQ^*)$ in which both families have size at least $f^\comp{2t}(1)$ and are arc-disjoint with elements of $\cH^+$. 
Now applying \Cref{cor:multiple-in} to $D$, $(\cP^*,\cQ^*)$, $n=1$ and $2t$ (in place of $t$), we find a family $\cH^-=\{\pi^-_j\}_{j=1}^{2t}$ of $2t$ arc-disjoint immersion models of onions with sink $x$ and mutually different sources $\{z_j\}_{j=1}^{2t}$. Moreover, $A(\cH^+)\cap A(\cH^-)=\emptyset$. It now remains to note that it is possible to find two sets $I,J\subseteq [2t]$ such that $|I|=|J|=t$ and $\{y_i : i\in I\}\cap\{z_j : j\in J\}=\emptyset$. 
The union of immersions $\pi^+_i$ where $i\in I$ and $\pi^-_j$ where $j\in J$ forms an immersion of a \os{t}.
\end{proof}

\section{Proof of \texorpdfstring{\cref{thm:no-cut}}{Theorem 1.3}}\label{sec:proof}
In this section we prove \Cref{thm:onion-harvesting} and then \Cref{thm:no-cut}. 

For a digraph $D$, a vertex $y \in V(D)$ and a set $Z \subseteq V(D) \setminus \{y\}$ we denote by $\mu(y,Z)$ ($\mu(Z,y)$, respectively) the maximum number of arc-disjoint paths from $y$ to $Z$ (from $Z$ to $y$, respectively) that can be found in $D$.

Since our proof of \cref{thm:no-cut} requires an iterative application of \Cref{thm:onion-harvesting}, to simplify this procedure we actually prove the following, slightly more general version of the latter one.
Note that \Cref{thm:onion-harvesting} follows from \Cref{lem:onion-harvesting} by taking $Y=\{y\}$.

\begin{theorem}\label{lem:onion-harvesting}
There exists a function $g\colon \N\times \N\to \N$ such that the following holds.
 Suppose $D$ is a digraph, $y \in V(D)$, and $Z$ is a subset of $V(D) \setminus\{y\}$.
Let $\cP$ and $\cQ$ be maximum families of pairwise arc-disjoint simple paths such that elements of $\cP$ start at $y$ and end at $Z$, and elements of $\cQ$ start at $Z$ and end at $y$.
Assume that $|\cP|>g(t,k)$ and $|\cQ|>g(t,k)$, for some $t,k\in \N$. Then at least one of the following holds:
 \begin{itemize}[nosep]
  \item $D$ contains the $t$-onion star as an immersion;
  \item there is a subset $\cP'$ of $\cP$, and a subset $\cQ'$ of $\cQ$, each consisting of~$k$ paths, such that the elements of $\cP' \cup \cQ'$ are pairwise arc-disjoint.
 \end{itemize}
\end{theorem}
\begin{proof}[Proof of \Cref{lem:onion-harvesting}]
We define \[
	g(t,k)=2^N(N-1)+1
\] where $N=\max\{k,F(t)\}$ and $F$ is the function given by \Cref{lem:wc}. 
By \Cref{thm:thomasen} applied to the intersection graph of $(\cP,\cQ)$,
there exist families $\cP' \subseteq \cP$ and $\cQ' \subseteq \cQ$ of paths such that $|\cP'|,|\cQ'|\geq N$ and either $(\cP',\cQ')$ is a well-crossing pair, or  $A(\cP')\cap A(\cQ')=\emptyset$.
In the first case, we conclude by \Cref{lem:wc} the existence of an immersion model of a $t$-onion-star in $D$.
As of the second outcome, it is precisely as desired.
\end{proof}

We are now ready to prove \Cref{thm:no-cut}. 

\begin{proof}[Proof of \Cref{thm:no-cut}] 
Define function $f$ as
\[f(t)=2t\cdot g^\comp{4t^2}_t(2),\]
where $g_t(k)\coloneqq g(t,k)$ is the function from \Cref{lem:onion-harvesting}.

We arbitrarily enumerate the vertices of $X$ as $y,x_1,\ldots,x_{2t}$. 
Consider an auxiliary digraph $D'$ obtained from $D$ by adding a single vertex $v$, $\frac{f(t)}{2t}$ arcs $vx_i$ and $\frac{f(t)}{2t}$ arcs $x_iv$ for every $i \in [2t]$.
First, observe that in $D'$ there exists a family of $f(t)$ arc-disjoint $(v,y)$-paths.
Indeed, let $\cP'$ be the maximum family of arc-disjoint $(v,y)$-paths, and let $k$ be the size of $\cP'$.
By \Cref{thm:menger}, there exists a $(v,y)$-cut $(A,B)$ of size $k$. 
Since $|\delta^+(v)|=f(t)$, we have $k \leq f(t)$. 
Observe that if $k<f(t)$, then $A\setminus \{v\} \neq \emptyset$, which implies that there exists $i \in [2t]$ such that $x_i \in A$. 
However, in such a case, $(A\setminus\{v\},B)$ is an $(x_i,y)$-cut in $D$ of size $k < f(t)$ contradicting the assumption of the theorem.
Hence, $|\cP'| = f(t)$.
By a symmetric argument we obtain the existence of a family $\cQ'$ of $f(t)$ arc-disjoint $(y,v)$-paths in $D'$.
Clearly, we can assume that $\cP'$ and $\cQ'$ consist of simple paths.

Let $\cP$ and $\cQ$ respectively denote the families of paths $\cP'$ and $\cQ'$ once every path has been restricted to the arcs of $D$.
Note that each $P \in \cP$ (resp.~$Q \in \cQ$) is an $(x,y)$-path (resp.~$(y,x)$-path) for some $x \in X \setminus \{y\}$.
Moreover, for every $i \in [2t]$, the number of $(x_i,y)$-paths in $\cP$ (resp.~$(y,x_i)$-paths in $\cQ$) is precisely $\frac{f(t)}{2t}$.
Hence, if we denote by $\cP_i$ the subset of $\cP$ that consists of $(x_i,y)$-paths, for every $i \in [2t]$, we have that $|\cP_i|=\frac{f(t)}{2t}=g^\comp{4t^2}_t(2)$.
Similarly, by $\cQ_i$ we denote the subset of $\cQ$ that consists of $(y,x_i)$-paths, and for every $i \in [2t]$ we have that $|\cQ_i|=g^\comp{4t^2}_t(2)$.

Let $p \in [4t^2]$.
We fix an arbitrary order on the elements of $[2t] \times [2t]$, and for the $p$-th element $(i,j)$ we proceed as follows.
We apply \Cref{lem:onion-harvesting} to families $\cP_i$ and $\cQ_j$, setting $Z=\{x_i,x_j\}$, $y$, and $k=g^\comp{4t^2-p}_t(2)$.
Either we conclude that there exists an immersion model of a $t$-onion-star in $D$, or replace $\cP_i$ and $\cQ_j$ by their subsets $\cP'_i$ and $\cQ'_j$ of size $g^\comp{4t^2-p}_t(2)$.
Clearly the theorem holds if the first case occurs.
In the other case, for every $i' \neq i, j' \neq j$, we reduce sets $\cP_{i'}$ and $\cQ_{j'}$ arbitrarily so that each of them contain exactly $g^\comp{4t^2-p}_t(2)$ paths, and proceed to the next iteration.

After $4t^2$ iterations, if no onion-star immersion model is found, for every $i$ the families $\cP_i$ and $\cQ_i$ have precisely $g^\comp{0}(2)=2$ elements each.
Let $\cS_i=\cP_i \cup \cQ_i$ and $\cS=\bigcup_{i\in [2t]} \cS_i$. 
We observe that paths from $\cS$ are now pairwise arc-disjoint, and each $\cS_i$ contains two paths $x_i \to y$ and two paths $y \to x_i$ for every $i \in [2t]$.
Hence, there exists an $t$-onion-star immersion model in~$D$, with center being $y$.
This concludes the proof.
\end{proof}

\bibliographystyle{alpha}
\bibliography{main}

\end{document}